\documentclass[11pt]{article}
\usepackage[utf8]{inputenc}
\usepackage[T1]{fontenc}
\usepackage{charter}
\usepackage{amsmath, amsthm, amssymb}
\usepackage{stmaryrd}
\usepackage{hyperref}
\usepackage{cleveref}
\usepackage{tikz}
\usepackage{tikz-cd}
\newcommand{\ku}{k(\!(u)\!)}
\newcommand{\kubr}{k[\![u]\!]}

\newcommand{\upar}{(\!(u)\!)}
\newcommand{\id}{\text{id}}
\newcommand{\N}{\mathbf{N}}
\newcommand{\Z}{\mathbf{Z}}
\newcommand{\R}{\mathbf{R}}
\newcommand{\Q}{\mathbf{Q}}
\newcommand{\F}{\mathbf{F}}
\newcommand{\Ksep}{K^{\text{sep}}}
\newcommand{\phimod}{\text{Mod}^\phi_{/K}}
\newcommand{\phimodet}{\text{Mod}^\phi_{/K,\text{et}}}

\title{Factorization of skew polynomials over $\ku$}
\author{J\'er\'emy Le Borgne}

\newtheorem{theorem}{Theorem}[section]
\newtheorem{definition}[theorem]{Definition}
\newtheorem{proposition}[theorem]{Proposition}
\newtheorem{lemma}[theorem]{Lemma}
\newtheorem{corollary}[theorem]{Corollary}
\newtheorem{remark}[theorem]{Remark}
\newtheorem{example}[theorem]{Example}

\begin{document}
\author{Jérémy Le Borgne\thanks{Univ Rennes, IRMAR - UMR 6625, F-35000 Rennes, France}}

\date{September 25, 2022}
\maketitle

\begin{abstract}
    Let $k$ be a perfect field of characteristic $p>0$, and let $K = \ku$ be the field of Laurent series over $K$. We study the skew polynomial ring $K[T,\phi]$, where $\phi$ is an endomorphism of $K$ that extends a Frobenius endomorphism of $k$. We give a description of the irreducible skew polynomials, develop an analogue of the theory of the Newton polygon in this context, and classify the similarity classes of irreducible elements.
\end{abstract}
\tableofcontents
\section*{Introduction}
Let $p>0$ be a prime number and let $k$ be a perfect field of characteristic $p$. Let $K = \ku$ be the field of formal power series with coefficients in $k$. We denote by $v$ the valuation map on $K$. Let $\sigma$ be a power of the Frobenius automorphism of $k$ (possibly, $\sigma = \id$). Let $b\geq 2$ be an integer. The field $K$ is endowed with the endomorphism $\phi$ : $K \to K$ defined by:
\[\phi\left( \sum_{n \geq n_0} a_n u^n\right) = \sum_{n \geq n_0} \sigma(a_n)u^{bn}.\]
The ring of skew polynomials over $K$ with endomorphism $\sigma$, denoted by $K[T,\phi]$, is the $K$-vector space $K[T]$ endowed with the multiplication rule defined by $Xa = \phi(a)X$ for $a \in K$. In the general setting, such rings were introduced and studied by Ore (\cite{zbMATH03009786} and have led to extensive literature both for their theoretical study (\cite{zbMATH00960149}), practical arithmetics (see \cite{zbMATH01216480}, \cite{zbMATH07245215}), and applications (\cite{zbMATH03939252}, \cite{zbMATH05616450}) .\\
The paper is divided in five main sections. In the second section, we give an overview of the classical and fundamental case where $\phi$ is the Frobenius endomorphism and $K = \F_p\upar$ using the classical theory of the Newton polygon and the relation between skew polynomials and Galois representations, which we classify in this case. In the third section, we introduce a theory of the Newton polygon for skew polynomials over $K$ in the general case, where it is not possible use representation theory. In the fourth section, we give a description of irreducible skew polynomials. Namely, Theorem \ref{thm:classif-irred} shows that the irreducible skew polynomials are those whose Newton polygon has a single slope (which we call monoclinic) and whose corresponding reduction in $k[T,\sigma]$ is irreducible (in a slightly twisted way). In the fifht section, we give a classification of the similarity classes of irreducible polynomials, showing that each class can be described by the data of a slope and an irreducible element in $k[T, \sigma]$ and giving the conditions for similarity of these classes (see Proposition \ref{prop:irred-pmu} and Proposition \ref{prop:p1p2} for a detailed formulation). 

\section{Skew polynomials and $\phi$-modules over $K$}

Let $p>0$ be a prime number and let $k$ be a perfect field of characteristic $p$. Let $K = \ku$ be the field of formal power series with coefficients in $k$. Let $\sigma$ be a power of the Frobenius automorphism of $k$ (possibly, $\sigma = \id$). Let $b\geq 2$ be an integer. The field $K$ is endowed with the endomorphism $\phi$ : $K \to K$ defined by:
\[\phi\left( \sum_{n \geq n_0} a_n u^n\right) = \sum_{n \geq n_0} \sigma(a_n)u^{bn}.\]
One important example is the case when $\phi$ is the Frobenius endomorphism of $K$, $x \mapsto x^p$, for which we give a presentation in Section \ref{sub:classical} which takes advantage of the links between $\phi$-modules and Galois representations.

\begin{definition}
The ring of skew polynomials over $K$ with endomorphism $\phi$ is the ring $K[T,\phi]$. The elements of $T$ are the same as elements of $K[T]$, and multiplication is determined by the formula $Ta = \phi(a)T$ for all $a \in K$.
\end{definition}
\begin{example}
If $k = \F_{p^2} = \F_p(\alpha)$, endowed with the Frobenius endomorphism $\phi$, then in $K[T,\phi]$ one has:
\[ (T^2+uT+1)(\alpha  T+ 1+u) = \alpha T^3 + (1+\alpha^p u + u^{p^2}) T^2 + (\alpha + u + u^p) T + 1+u.
\]  
\end{example}
The ring $K[T,\phi]$ is noncommutative, but shares some nice properties with $K[T]$. In particular, it is right-euclidean, and thus has a notion of irreducible elements and a factorization theorem.
\begin{theorem}[Ore, \cite{zbMATH03009786}]\label{thm:ore}
Let $P \in K[T,\phi]$, then there exist $P_1, \dots, P_r \in K[T,\phi]$ irreducible elements such that $P = P_1\cdots P_r$.
\end{theorem}
Such a factorization is not unique in general, and describing how two given factorizations are related is easier to do in the language of $\phi$-modules. As for $K[T]$-modules in linear algebra, a module over $K[T,\phi]$ corresponds to a vector space endowed with an endomorphism, but in this case the endomorphism is merely \emph{semilinear} with respect to $\phi$.
\begin{definition}
A $\phi$-module over $K$ is a couple $(D,\phi_D)$ where $D$ is a finite dimensional vector space over $K$, and $\phi_D$ : $D \to D$ a $\phi$-semilinear endomorphism.
\end{definition}

A morphism of $\phi$-modules is a $K$-linear map $f$ : $D_1 \to D_2$ such that the following diagram is commutative:
\[\begin{tikzcd}
	{D_1} & {D_2} \\
	{D_1} & {D_2}
	\arrow["f", from=1-1, to=1-2]
	\arrow["f", from=2-1, to=2-2]
	\arrow["{\phi_{D_1}}"', from=1-1, to=2-1]
	\arrow["{\phi_{D_2}}", from=1-2, to=2-2]
\end{tikzcd}.\]
The set of $\phi$-modules over $K$ forms an abelian category which we denote by $\phimod$. We aim to study the full subcategory $\phimodet$ of {\'e}tale $\phi$-modules over $K$, whose objects are the $\phi$-modules $(D,\phi_D)$ such that $\phi_D(D)$ spans the $K$-vector space $D$.\\
Alternatively, let $R = K[T,\phi]$ be the noncommutative $K$-algebra of skew polynomials. Then $\phimod$ is equivalent to the category of $\text{Mod}_R$ of left-$R$-modules that have finite dimension over $K$ (the semilinear map $\phi_D$ corresponds to the map of left-multiplication by the indeterminate $T$).
 If $\mathcal B = (e_1, \dots, e_d)$ is a basis of the $\phi$-module $D$, and $(e_1^*, \dots, e_d^*)$ is the dual basis, then the matrix of $\phi_D$ in the basis $\mathcal B$ is the matrix $M \in \mathcal{M}_d(K)$ whose coefficient in position $(i,j)$ is $e_i^*(\phi_D(e_j))$. In the case, if $P \in GL_d(K)$, the matrix of $\phi_D$ in the basis given by the columns of $P$ is $P^{-1}M\phi(P)$, where $\phi(P)$ is the matrix obtained from $P$ by applying $\phi$ to each coefficient. In particular, since $b\geq 2$, there exists a basis $\mathcal B$ of $D$ such that $\kubr$ submodule generated by the basis is stable by $\phi_D$, \emph{i.e.} such that the matrix of $\phi_D$ is this basis has coefficients in $\kubr$.\\
This point of view allows us to describe more precisely how two factorizations of $P \in K[T,\phi]$ (as in Theorem \ref{thm:ore}) are related.
\begin{definition}
Let $A,B \in K[T,\phi]$, then $A$ and $B$ are similar if the corresponding $\phi$-modules $K[T,\phi]/K[T,\phi]A$ and $K[T,\phi]/K[T,\phi]B$ are isomorphic.
\end{definition}
Then, the number of occurrences of each similarity class of irreducible skew polynomials that appear in a factorization of a given skew polynomial $P$ does not depend on the factorization, but only on $P$ itself.\\
The aim of this paper is to give a description of irreducible elements of $K[T,\phi]$, give a classification of similarity classes, and describe how the similarity classes of the irreducible factors of a skew polynomial can be determined.\\
On a broader scope, we also aim to set the theoretical foundations to give a factorization algorithm for elements of $K[T, \phi]$, which is planned in future work.
\section{The classical case}\label{sub:classical}
The classical case for $\phi$-modules is the case when $\phi$ is (a power of the) Frobenius morphism. For the sake of simplicity, we will assume that $\phi(x) = x^p$ for all $x \in K$. Let $\Ksep$ be a separable closure of $K$, then $\Ksep$, which is endowed with the canonical Frobenius morphism $x \mapsto x^p$ that we also denote by $\phi$. Let us recall how the classical theories of $\phi$-modules (see \cite{zbMATH03425708}) and Newton polygons (see \cite{zbMATH03957242} , Chap. 6) apply in this case.
\subsection{Skew polynomials and $\phi$-modules over $K$}
Let $(D,\phi_D)$ be an étale $\phi$-module over $K$. Then $\text{Hom}_{K,\phi}(D,\Ksep)$ is a $\F_p$-vector space. Moreover, the Galois group $\mathcal{G} = \text{Gal}(\Ksep/K)$ acts on $\text{Hom}_{K,\phi}(D,\Ksep)$, and $V$ is invariant under this action. Therefore, $V$ is naturally a $\F_p$-linear representation of $\mathcal{G}$. Conversely, if $V$ is a $\F_p$-representation of $\mathcal{G}$, then $\text{Hom}_{\mathcal{G}}(V,\Ksep)$ is an étale $\phi$-module over $K$, and these two constructions are converse of each other (thus, the corresponding functors are equivalences of categories between $\phi$-modules over $K$ and $\F_p$-representations of $\mathcal{G}$).\\
\[
\begin{array}{rcl}
\{\F_p-\text{representations of }\mathcal{G}_K\
 &\longrightarrow  & \{\text{{\'E}tale }\phi\text{-modules over }K\}\\
V & \mapsto & \text{Hom}_{{\mathcal G}_K}(V, \Ksep)\\
\text{Hom}_{K, \phi}(D, \Ksep) &\mapsfrom &D
\end{array}
\]
Now, let $P \in K[T,\phi]$ and let $D_P = K[T,\phi]/K[T,\phi]P$. Then the corresponding Galois representation is the set $V_P$ of roots of $P(\phi)$ in $\Ksep$ (more precisely, the map $D_P \to V_P$ defined by $f \mapsto f(t)$ is an isomorphism of $\F_p$-representations of $\mathcal{G}$ from $V$ to $V_P$).
\begin{proposition}
Let $V_P \subset \Ksep$ be the $\F_q$-vector subspace of roots of $P(\phi)$. Then $V$ has a nondecreasing filtration $(V_{\mu})_{\mu \in \R}$ of $V$ by subrepresentation. The jumps of the filtration are the opposite of the valuations of the elements of $V$.
\end{proposition}
\begin{proof}
 By the classical theory of the Newton polygon, the valuations of the roots of the linearized polynomial $P(\phi)$ can be recovered from its Newton polygon. Let $\mu \in \R$ and let $V_\mu = \{x \in V,~v(x) \geq -\mu\}$. Then $V_\mu$ is a $\F_q$-subspace of $V$ that is stable under the action of $\mathcal{G}$, \emph{i.e.} a subrepresentations of $V$. Therefore, the family $(V_\mu)_{\mu \in \R}$ is an increasing filtration of $V$. Let $V_\mu^+ = \{x \in V,~v(x) >-\mu\}$, then $V_\mu \neq V_{\mu}^+$ if and only if $P(\phi)$ has a root of valuation $-\mu$. Thus, the slope filtration has finitely many jumps that are the opposite of the valuations of the roots of $P(\phi)$. In general, a factor of $P(\phi)$ of given valuation does not correspond to a subrepresentation of $V$ (only to a subquotient of $V$), but a right factor does.\\
 If the polynomial $P \in K[T,\phi]$ is irreducible, then so is the representation $V_P$, and therefore the Newton polygon of $P(\phi)$ only has one slope. The corresponding filtration on the $\phi$-module $D_P$ yields a factorization of the skew polynomial $D_P$ as a product of skew polynomials with a single slope, which may not be irreducible.
\end{proof}
These results are generalized in \cite{Car09} to the case of a more general endomorphism $\phi$, when $k$ is algebraically closed. The first section of the present paper is to further generalize these results simultaneously in two directions: first, to the case of a perfect residue field, and second to the more general version of the endomorphism $\phi$ that has already been introduced.
\subsection{Irreducible Galois representations}
In order to highlight what type of result should be expected for the first generalization that we indend to give, we present an elementary version in the classical case (where $\phi$ is the Frobenius endomorphism) based on the classification of irreducible Galois representations and the equivalence of categories of the previous section.\\
In this subsection, we assume that $k=\F_p$. Let $K = \ku$, $\Ksep$ be the separable closure of $K$ and let ${\mathcal G}_K = \text{Gal}(\Ksep/K)$ be the absolute Galois group of $K$. Recall that this group has a natural ramification filtration, which cuts out a filtration of $\Ksep$ by subextensions whose first steps are $\Ksep \supset K^\text{tr} \supset K^{\text{ur}}\supset K$. Let $I_t = \text{Gal}(K^\text{tr}/K^\text{ur})$ be the tame inertia subgroup. It is an abelian group and for all $n \geq 1$, the fundamental character of level $n$ is defined as:
\[
\omega_n : \left|\begin{array}{ccc}
I_t & \to & \bar{\F}_p^\times\\
g & \mapsto & \frac{g u_n}{u_n},
\end{array}\right.
\]
where $u_n = u^{\frac{1}{p^n - 1}}$. Further, let $\sigma \in \mathcal{H}_K$ such that the projection of $\sigma$ in $\text{Gal}(K^\text{ur}/K) \simeq \text{Gal}(\bar{\F}_p/\F_p)$ is the Frobenius map $x \mapsto x^p$. 
\begin{proposition}\label{prop:Fprep}
The irreducible $\F_p$-representations of ${\mathcal G}_K$ are described by:
\begin{itemize}
\item a level $\delta$ and an integer $0 \leq s \leq p^\delta -1$,
\item an irreducible skew polynomial $P = T^m - \sum_{i=0}^{m-1} \lambda_i T^i \in \F_{p^\delta}[T,\sigma]$.
\end{itemize}
The corresponding representation has a basis over $\F_{p^\delta}$ that is of the form $(x, \sigma x, \dots, \sigma^{m-1}x)$ such that:
\begin{itemize}
\item $\forall g \in I_t$, $gx = \omega_\delta^s(g)x$,
\item $\sigma^m x = \sum_{i=0}^{m-1} \lambda_i \sigma^i x$,
and $V$ only depends up to isomorphism on the digits of $s$ in base $p$ and the similarity class of $P$.
\end{itemize}
\end{proposition}
\begin{proof}
Let $V$ be an irreducible representation of ${\mathcal G}_K$. Let $I_p \subset {\mathcal G}_K$ be the wild inertia sugroup: since $V$ is irreducible, the action of ${\mathcal G}_K$ factors through $I_p$ and yields an irreducible representation of $\mathcal{H}_K = {\mathcal G}_K/I_p$. Now, let $W \subset V$ be an irreducible representation of $I_t$. Then $E = \text{End}_{I_t}(W)$ is a division algebra by Schur's lemma, hence a field by Wedderburn's theorem. This endows $W$ with a natural structure of $E$-representation. Since $I_t$ is abelian, the elements of $I_t$ act as elements of $E$, so that $W$ is has dimension 1 as a $E$-vector space. Let $\delta = \dim W$, then $W$ can be identified with a character $\omega$ : $I_t \to \F_{p^\delta}^\times$. Since $I_t$ is a procyclic group, such characters are precisely the powers of the fundamental character of level $\delta$, \emph{i.e.} there exists an integer $0 \leq s \leq p^\delta - 1$ such that $\omega = \omega_\delta^s$. We may see as a $\F_p$-representation through the choice of a basis of $\F_{p^\delta}/\F_p$: any other choice of basis may lead to a different value of $s$, but the digits of $s$ in base $p$ does not depend on this choice. These digits are the tame inertia weights of the representation $W$.\\
It remains to see how the Frobenius acts on $V$. Assume $W \simeq \omega_{\delta}^s$, then $\sigma W$ is stable by $I_t$ and, as a representation of $I_t$, is isomorphic to $\omega_\delta^{ps}$ (because of the relation $\sigma g \sigma^{-1} = g^p$ for all $g \in I_t$). Let $m$ be maximal such that the sum $W_m = W + \sigma W + \cdots + \sigma^{m-1}W$ is a direct sum of $\F_{p^\delta}$-vector spaces. Then by hypothesis, $\sigma^r W \cap W_r$ is nonzero, so that $\sigma^r W \subset W_r$. Therefore, $W_m$ is stable under the action of $I_t$ and $\sigma$, and thus $W_m = V$. One has:
\[V = \omega_{\delta}^s \oplus \omega_{\delta}^{ps} \oplus \cdots \oplus \omega_{\delta}^{p^{m-1}s}. \]
Let $x \in W$, then there exist $\lambda_0,\dots,\lambda_{r_1}$ such that $\sigma^r x = \lambda_0 x + \lambda_1 \sigma x + \cdots + \lambda_{r-1}\sigma^{m-1}x$. The polynomial $T^m - \sum_{i=0}^{m-1} \lambda_i T^i$ depends on the choice of $W$ and $x$, but its similarity class in $E[T,\sigma]$ does not. Indeed, the map $\sigma^\delta$ is $E$-linear and its minimal polynomial is the norm $\mathcal{N}(P)$, which does not depend on any choice and completely determines the similarity class of $P$ (see \cite{zbMATH06680379}, Prop. 2.1.17.).
\end{proof}

\subsection{Irreducible skew polynomials in $K[T,\phi]$ (classical case)}
We still assume that $K = \F_p\upar$. According to the previous subsections, an irreducible element $A \in K[T,\phi]$ corresponds to an irreducible representation of $\mathcal{G}_K$, and therefore its similarity class can be described by a level $\delta$, a weight $s$, and an irreducible polynomial $Q \in \F_p[T^\delta]$. Let us now explain how these invariants may be computed directly from $A$.
Let $s \in \Z$, let $\delta \in \N^*$ and let $m \in \N^*$. Let $v = -\frac{p^{m\delta}-1}{p^{\delta} -1} \in \N$. Let $\lambda_1, \dots, \lambda_{m-1} \in k$ and $\lambda_0 \in k^\times$.
\begin{proposition}\label{prop:repres_assoc}
Let $A= (T^{\delta m} - \lambda_{m-1} T^{\delta (m-1)} - \cdots -\lambda_0)u^{sv} \in K[T^\delta,\phi]$. Then 
 there exists $x \in V$ such that for $g \in I_t$, $gx = \omega_\delta^s(g)x$ and such that $\sigma^{\delta m} x = \sum_{j=0}^{m-1} \lambda_{j}\sigma^{\delta j}x$.
\end{proposition}
\begin{proof}
By definition, the corresponding representation of $\mathcal{G}_K$ is the set of roots of
\[ L_A = u^{p^{\delta m }sv}x^{p^{\delta m}} - u^{p^{\delta (m-1) }sv}\lambda_{m-1}x^{p^{\delta(m-1)}} - \cdots -\lambda_0u^{sv}x.\]
The Newton polygon of this polynomial has only one slope, equal to $-\frac{s}{p^\delta-1}$, so that all the nonzero roots of this polynomial have the same valuation, equal to $\frac{s}{p^\delta - 1}$. More precisely, if $\xi$ is a root of $A$, then $u^{-\frac{s}{p^\delta -1}}\xi$ is a root of $x^{p^{\delta m}} - \sum_{j=0}^{m-1}\lambda_j x^{p^{\delta j}}$, so that $x^{\delta m} - \sum_{j=0}^{m-1} \lambda_i x^{\delta i}$ is the minimal polynomial of $\sigma$ over $\F_p$. 
\end{proof}
\begin{corollary}
Every irreducible {\'e}tale skew polynomial in $K[T,\phi]$ is of the form $\bar A u^{sv}$ where $\bar A \in k[T^\delta]$ is irreducible of degree $m$ (for some $m$), $s \in \Z$ and $v = \frac{p^{m\delta} - 1}{p^\delta - 1}$.
\end{corollary}
\begin{proof}
By Proposition \ref{prop:repres_assoc}, each such polynomial corresponds to an irreducible representation of $\mathcal{G}_K$ and by Proposition \ref{prop:Fprep} every irreducible representation of $\mathcal{G}_K$ is of this form.
\end{proof}
\section{The Newton polygon of a skew polynomial}
From now on, we revert to the more general setting where $k$ is perfect of characteristic $p>0$, $K = \ku$ and $\phi$ is such that $\phi(u) = u^b$ with $b \geq 2$ an integer, and the restriction of $\phi$ to $K$ is a power of the Frobenius endomorphism. The goal of this section is to introduce a notion of Newton polygons for elements of $K[T,\phi]$ that is analogous to the classical theory.
\subsection{Definitions}
\begin{definition}
Let $A \in K[T,\phi]$, $A = \sum_{i=0}^d a_i T^i$. The \emph{Newton polygon of $P$} is the lower convex hull of the set of points $\{(b^i, v(a_i)),~0\leq i \leq d\} \subset \R^2\}$.
\end{definition}
\begin{definition}
The \emph{slopes} of the skew polynomial $A \in K[T,\phi]$ are the slopes of its Newton polygon.\\
The \emph{multiplicity} of a slope $\mu$ of $A$ is $m$ if the endpoints of the segment of slope $\mu$ in the Newton polygon of $A$ are of the form $(b^i, v(a_i)), (b^{i+m}, v(a_{i+m}))$.\\
The skew polynomial $A$ is said to be $\emph{monoclinic}$ if it has only one slope.
\end{definition}
\begin{definition}
Let $A \in K[T,\phi]$ and let $\mu$ be a slope of $A$. Then there exists a unique $\nu \in \Q$ such that $u^\nu Au^{-\mu}$ has valuation zero. The reduction of $u^\nu A u^{-\mu}$ modulo the ideal of skew polynomials of positive valuation is called the \emph{$\mu$-reduction} of $A$.
\end{definition}

\subsection{The slopes of a product of skew polynomials}
\begin{lemma}\label{lem:slopes-prod}
Let $P,Q \in K[T,\phi]$, with $P = \sum_{i=0}^n p_iT^i$, and $Q = \sum_{j=0}^d q_jT^j$ such that:
\begin{itemize}
    \item $\min_{0 \leq i \leq n}v(p_i) = 0$,
    \item $q_d = 1$, $v(q_0) = 0$, and $v(q_i)\geq 0$ for all $1 \leq i \leq d-1$.
\end{itemize}
Let $i_0 = \min\{0 \leq i \leq n,~v(p_i) = 0\}$ and $i_1 = \max\{0 \leq i \leq n,~v(p_i) = 0\}$. Then the vertices of the Newton polygon of $A = PQ$ are:
\begin{itemize}
    \item $(b^i, v(p_i))$ if $i \leq i_0$ and $(b^i, v(p_i))$ is a vertex of the Newton polygon of $P$,
    \item $(b^{i+d},v(p_i))$ if $i\geq i_1$ and $(b^i, v(p_i))$ is a vertex of the Newton polygon of $P$.
\end{itemize}
In particular, the Newton polygon of $PQ$ has an edge of slope $0$ and multiplicity $i_1 - i_0 + d$.
\end{lemma}
\begin{proof}
Let $0 \leq i \leq i_0$. Then the coefficient of $T^i$ in $PQ$ is $c_i=\sum_{j=0}^i p_j\phi^j(q_{i-j})$. For $ 0 \leq j < i$, $v(p_j) \geq 0$, $v(q_{i-j}) \geq 0$, and  since $v(q_0) =0$, $v(c_i) \geq v(p_i)$. Moreover, assume that $(b^i, v(p_i))$ is a vertex of the Newton polygon of $P$. Then for all $0 \leq j \leq i$, $v(p_j) > v(p_i)$, so that $v(c_i) = v(p_i)$. Moreover, if $(b^i, v(p_i))$ and $(b^{i'}, v(p_{i'}))$ are two consecutive vertices of the Newton polygon of $P$, then for $i \leq j \leq i'$, $\frac{v(p_j) -  v(p_{i'})}{b^j - b^{i'}}\geq \frac{v(p_i) - v(p_{i'})}{b^i - b^{i'}}$. Thus, $\frac{v(c_j) -  v(c_{i'})}{b^j - b^{i'}}\geq \frac{v(c_i) - v(c_{i'})}{b^i - b^{i'}}$. This shows that, for $i \leq i_0$, the vertices of the Newton polygon of $PQ$ are the same as the vertices of the Newton polygon of $P$.\\
On the other hand, for $i \geq i_1$, $c_{i+d} = \sum_{j=i}^{i+d-j} p_j\phi^j(q_{i+d-j})$. This shows in particular that $v(c_{i_1+d}) = 0$, and more generally that $v(c_{i+d}) = v(p_i)$ whenever $(b^i, v(p_i))$ is a vertex of the Newton polygon of $P$. Then a calculation similar to the previous one shows that the $(b^{i+d}, v(p_i))$ for the $i \geq i_1$ such that $(b^i, v(p_i))$ is a vertex of the Newton polygon of $P$, are indeed vertices of the Newton polygon of $PQ$.\\
Finally, since $v(c_{i_0}) = v(c_{i_1+d}) = 0$, there are no other vertices in the Newton polygon of $PQ$.
\end{proof}
\begin{proposition}\label{prop:slopes-prod}
Let $P,Q \in K[T,\phi]$. Let $\mu_1< \cdots < \mu_r$ be the slopes of $P$. Assume that $Q$ is monic and monoclinic of slope $\mu$ and degree $d$, so that $\mu = -\frac{s}{b^d-1}$. Let $r_0$ be such that $\mu_{r_0} \leq \mu < \mu_{r_0+1}$. Then the slopes of $PQ$ are $\mu_1+s< \cdots < \mu_{r_0}+s \leq \mu< b^{-d}\mu_{r_0+1}< \cdots < b^{-d}\mu_r$.
%Moreover, the $\mu$-reduction of $PQ$ is $P_\mu Q_\mu$, where $P_\mu, Q_\mu$ are the respective $\mu$-reductions of $P,Q$.
\end{proposition}
\begin{proof}
Multiplying by $u^{-\mu}$ on the right translates the slopes by $-\mu$, so the result on the slopes follows from applying Lemma \ref{lem:slopes-prod} to $Pu^{-b^d \mu}$ and $u^{b^d \mu}Qu^{- \mu}$ (note that slopes may be computed in any extension of $K$ as they are defined by the Newton polygon).\\
Indeed, the lemma shows that the smallest slopes of $Pu^{-b^d \mu} (u^{b^d \mu} Q u^{-mu} $ are the $\mu_i - b^d \mu$, and its highest slopes are the $(\mu_i - b^d)/b^d$, with one middle slope equal to zero. Multiplying again on the right by $u^\mu$ allows us to recover the slopes of $PQ$, and since $(b^d - 1)\mu = -s$, we get the prescribed slopes.
\end{proof}

\section{Irreducible skew polynomials}
In this section, we use the results about Newton polygons of skew polynomials to give a description of the irreducible elements of $K[T,\phi]$.
\begin{definition}
Let $\mu \in \Q^\times$. Write $\mu = b^\alpha \frac{s}{t}$, with $\alpha \in \Z$, and the two integers $s \in \Z$ and $t\in \Z\setminus \{0\}$ coprime, and $t$ coprime to $b$. Then the $b$-length of $\mu$, denoted by $\ell_b(\mu)$, is the order of $b$ modulo $t$ (or $0$ if $t = 1$).
\end{definition}
\begin{remark}
By definition, when $\mu \in \Z_{(b)}$, $\ell_b(\mu)$ is the smallest integer $\ell$ such that $\mu(b^\ell -1)  \in \Z$. Every such integer is a multiple of $\ell_b(\mu)$.
\end{remark}
\begin{lemma}\label{lem:rightfactor}
Let $\mu$ be the smallest slope of $A \in K[T,\phi]$. Let $P_\mu$ be the $\mu$-reduction of $A$. Let $P$ be a right divisor of $P_\mu$ in $k[T^{\ell_b(\mu)}, \sigma^{\ell_b(\mu)}]$. Then there exists $Q_\mu \in K[T,\phi]$ such that $Q_\mu$ has slope $\mu$ and $\mu$-reduction $P$, and such that $Q_\mu$ is a right-divisor of $A$.
\end{lemma}
\begin{proof}
We may assume that $\mu =0$. We show that there are sequences of skew polynomials $(F_j)_{j\geq 0}$, $(G_j)_{j\geq 0}$, and a sequence $v_j$ with $\lim v_j = + \infty$, such that the valuation of $A-F_jG_j$ is $\geq v_j$, and such that $v(G_j - P) >0$.\\
By hypothesis, there exists $F \in k[T, \sigma]$, $R\in \kubr[T,\phi]$ and $\nu >0$ such that $A - FP= u^\nu R > 0$. Then the induction hypothesis holds for $j=0$ with $F_0 = F$, $G_0 = P$, and $\nu_0= \nu$.\\
Now assume that such sequences have been constructed up to some value of $j \geq 0$. We may write $A - F_jG_j = u^{\nu_j}R_j$. Let $\bar R_j$ be the reduction modulo $u^{>0}$ of $R_j$. Let $\bar R = MG_0 + a_0N$ be the right euclidean division of $\bar R$ by $G_0$, where $a_0 \in k^\times$ is the constant coefficient of $F$ (it is nonzero because $\mu$ is the smallest slope of $A$). Then we have $(F_j + u^{\nu_j}M)(G_j + u^{\nu_j}N) = F_jG_j + u^{\nu_j}MG_j + u^{\nu_j}a_0N + O(u^{\nu_j+1})$. Thus, if we let $F_{j+1} = F_j + u^{\nu_j} M$ and $G_{j+1} = G_j + u^{\nu_j}N$, then $v(A - F_{j+1}G_{j+1}) \geq \nu_{j}+1$.\\
In particular, the sequences $(F_j)$ and $(G_j)$ are convergent, and their respective limits $F$ and $G$ are such that $A = FG$ and $v(Gu^{-\mu} - P) > 0$. In particular, $Q_\mu = G$ is a right-divisor of $A$ and has slope $\mu$ and $\mu$-reduction $P$.  
\end{proof}
\begin{theorem}\label{thm:classif-irred}
Let $A \in K[T,\phi]$. Then $A$ is irreducible if and only if:
\begin{itemize}
    \item $A$ is monoclinic of slope $\mu$, with $\ell_b(\mu) = \ell$
    \item the $\mu$-reduction of $A$ is irreducible in $k[T^\ell, \sigma^\ell]$.
\end{itemize}
\end{theorem}
\begin{proof} By Lemma \ref{lem:rightfactor}, if $A$ has more than one slope or is monoclinic but has a $\mu$-reduction that is reducible, then $A$ is reducible. Conversely, assume that $A$ is monoclinic of slope $\mu$ and that its $\mu$-reduction is irreducible. Assume $A = A_1A_2$, then $u^{\nu}Au^{-\mu} = (u^\nu A_1u^{-\nu_2}) (u^{\nu_2} A_2u^{-\mu})$, where $\nu_2$ is such that $ (u^{\nu_2} A_2u^{-\mu})$ is the reduction of $A_2$. Then the irreducibility of the $\mu$-reduction of $A$ shows that either $A_1$ or $A_2$ is constant, which proves the irreducibility of $A$.
\end{proof}

Since the irreducible factors appearing in the factorization of a skew polynomial $A \in K[T,\phi]$ are determined up to similarity, we want to describe the similarity classes of irreducible skew polynomials. In terms of $\phi$-modules, this amounts to describing the isomorphism classes of simple objects in $\phimod$, and giving the similarity classes of the irreducible factors of a skew polynomial amounts to giving the semi-simplification of the corresponding $\phi$-module. Note that in general, the fact that $P$ is irreducible and appears in a factorization of $A$ does not guarantee that $A$ has a right-factor similar to $P$ (whereas this is indeed the case over a finite field, see \cite{zbMATH06680379}, Lemma 2.1.18).

\section{Similarity classes of irreducible skew polynomials}
The aim of this section is to describe the similarity classes of irreducible elements in $K[T,\phi]$. The slopes of a skew polynomial are not invariant by similarity: indeed, if $P \in K[T,\phi]$, then left-multiplication by $T$ followed by right-division by $T$ yield a polynomial similar to $P$ whose coefficients have had all their valuations mutiplied by $b$, so the same holds for its slopes. Similarly, multiplication by $u^m$ turns $P$ into a skew polynomial similar to $P$, and its slopes are the slopes of $P$ translated by $m$. This leads to the following defintions.

\begin{definition}
The rational numbers $\mu, \mu'$ are equivalent if $\ell_b(\mu) = \ell_b(\mu') = \ell$ and there exists an integer $0 \leq i \leq \ell$ such that $\mu - b^i \mu' \in \Z$.
\end{definition}
\begin{remark}
Alternatively, two rational numbers $\mu_1$ and $\mu_2$ are equivalent if and only if they have the same $b$-length $\ell$, and $(b^\ell -1)\mu_1$ and $(b^\ell -1)\mu_1$ have the same digits when they are written in base $b$.
\end{remark}

\begin{definition}
Let $\mu \in \Q^\times$ with $\ell_b(\mu) = \ell$ and let $P\in k[T^\ell, \sigma^\ell]$. Then $P^{[\mu]} = u^{-\mu}Pu^\mu \in K[T,\phi]$.
\end{definition}

\begin{lemma}
Let $\mu \in \Q^\times$ with $\ell_b(\mu) = \ell$ and $P\in k[T^\ell, \sigma^\ell]$, the skew polynomial $P^{[\mu]}$ has slope $\mu$ and $\mu$-reduction $P$.
\end{lemma}
\begin{proof}
The $\mu$-reduction can be computed directly from the definition of $P^{[\mu]}$.
\end{proof}

\begin{lemma}\label{lem:pmu}
Let $P \in K[T,\phi]$ be monic {\'e}tale with integral coefficients. Let $v_0 \geq 0$ be the valuation of the constant coefficient of $P$. Let $Q \in K[T,\phi]$ monic and assume that $v(P-Q)>bv_0/(b-1)$. Then $P$ and $Q$ are similar.
\end{lemma}
\begin{proof}
Let $C_P$ (resp. $C_Q$) be the companion matrix of $P$ (resp. of $Q$). Let $M_0 = I_d$ and define the sequence $(M_n)_{n \geq 0}$ inductively by $M_{n+1} = C_Q\phi(M)C_P^{-1}$. Since $v(\det C_P) = v_0$, the cofactor matrix formula shows that $v(C_QC_P^{-1} - I_d) > v_0/(b-1)$, \emph{i.e.} $v(M_1-M_0)>v_0/(b-1)$. Now assume that for some $n \in \N^{*}$, $v(M_n - M_{n-1})  = v > \frac{v_0}{b-1}$. Then $M_{n+1} - M_n = C_Q\phi(M_n - M_{n-1})C_{P}^{-1}$, and therefore $v(M_{n+1} - M_n) \geq bv - v_0 > v$. In particular, the sequence $(M_n)_{n \in \N}$ is convergent. Let $M$ be the limit of this sequence, then $M$ is nonsingular because it is equal to $I_n$ modulo $u$, and $C_P = M^{-1}C_Q\phi(M)$. Thus, $P$ and $Q$ are similar.
\end{proof}
\begin{lemma} \label{lem:p-sim-pmu} Let $P \in K[T,\phi]$ be monoclinic of slope $\mu$. Let $\bar{P}$ be its $\mu$-reduction. Then $P$ is similar to $\bar P^{[\mu]}$.
\end{lemma}
\begin{proof}
Let $Q = \bar P^{[\mu]} \in K[T,\phi]$ since $P$ is monoclinic of slope $\mu$. Up to left multiplication of $P$ by a constant, we may assume that $v(P - \bar P) >0$. Up to right mutliplication by a constant, we may assume that $\mu \leq 0$. By Lemma \ref{lem:pmu}, we may assume that the constant coefficient of $P$ is of the form $\lambda_0 u^s$, and in particular $P$ and $\bar P^{[\mu]}$ have the same constant coefficient.\\
Now let $D = K[T,\phi]/PK[T,\phi]$. This $\phi$-module is endowed with its canonical basis $(x, \phi(x), \dots, \phi^{d-1}(x))$, so that $P(\phi)(x) = 0$. Let $x_0 = x$, and define by induction $x_{n+1} = \bar P^{[\mu]} (\phi)(x_n) + x_n$. For $n \in \N^*$, $x_{n+1} - x_n = (\bar P^{[\mu]} - P)(\phi)(x_n - x_{n-1})$. Since $P$ and $\bar P^{[\mu]}$ have the same constant coefficient, this shows that the sequence $(x_n)_{n \in \N}$ is convergent. Its limit $x_\infty$ is such that $x_{\infty} = \bar P^{[\mu]} (\phi)(x_\infty) + x_\infty$, so that $\bar P^{[\mu]}(\phi)(x_\infty) = 0$. Therefore, $\bar P^{[\mu]}$ is similar to $P$.
\end{proof}

\begin{corollary}
The slopes of a product of skew polynomials, counted with multiplicities, are up to equivalence the slopes of the factors, counted with multiplicities.
\end{corollary}
\begin{proof}
Let $P,Q \in K[T,\phi]$. Let us show induction on the number of slopes of $Q$ that the slopes of $PQ$ are, up to equivalence, the slopes of $P$ and $Q$. Assume $Q$ is monclinic, then the result follows from Proposition \ref{prop:slopes-prod} (since $\mu_i + s$ and $\mu_i/b^d$ are equivalent to $\mu_i)$. Suppose the result holds when $Q$ has $m$ slopes, and suppose now that $Q$ has $m+1$ slopes. Then $Q$ may be factored as $Q = Q_1 Q_\mu$ with $Q_\mu$ monoclinic of slope $\mu$. By induction hypothesis, $Q_1$ has $m$ slopes, the slopes of $Q$ are the slopes of $Q_1$ and $\mu$, and the slopes of $PQ_1$ are the slopes of $P$ and $Q_1$ (up to equivalence). Thus, the results holds again as consquence of the case of monoclinic $Q$.
\end{proof}

\begin{proposition}\label{prop:irred-pmu}
Every irreductible element of $K[T,\phi]$ is similar to an element of the form $P^{[\mu]}$ with $P \in k[T^ \ell,\sigma^\ell]$ irreducible, with $\ell = \ell_b(\mu)$.
\end{proposition}
\begin{proof}
Let $P \in k[T,\phi]$ be irreducible. By Theorem \ref{thm:classif-irred}, $P$ is monoclinic of slope $\mu$ with $\ell_b(\mu) = \ell$, and its $\mu$-reduction $\bar P$ is irreducible in $k[T^\ell, \sigma^\ell]$.\\
By Lemma \ref{lem:p-sim-pmu}, $P$ is similar to $\bar P u^\mu$. Since $\bar P \in k[T^\ell, \sigma^\ell]$, $P$ is indeed of the prescribed form.
\end{proof}

\begin{proposition}\label{prop:p1p2}
The irreducible skew polynomials $P_1^{[\mu_1]}$ and $P_2^{[\mu_2]}$ are similar if and only if $\mu_1 \sim \mu_2$ and $P_1 \sim_{k[T^\ell, \sigma^\ell]} P_2$.
\end{proposition}
\begin{proof}
Let $P_1$ be a monoclinic monic skew polynomial, and let $C_1$ be its companion matrix. Then the slope of $P_1$ is $\frac{v(\det(C_1))}{b^{\deg P} - 1}$. Assume $P_2$ is similar to $P_1$, then there exists $M \in GL_d(K)$ such that $MC_1 = C_2 \phi(M)$, so that $\mu_1 + \det M = \mu_2 + b\det M$. Thus, $\mu_1 - \mu_2 \in (b-1)\Z$, so that $(b^d-1)\mu_1$ and $(b^d -1) \mu_2$ have the same digits when written in base $b$. This shows that $\mu_1 \sim \mu_2$. Thus, we may assume that $\mu_1 = \mu_2$, and in this case, the $\mu$-reductions of $P_1$ and $P_2$ are similar.\\
Conversely, if $\mu = \mu_1 \sim \mu_2$ and the $\mu$-reductions of $P_1$ and $P_2$ are similar, then an explicit isomorphism between the corresponding $\sigma$-modules over $k$ yields an isomorphism between the $\phi$-modules associated to $P_1$ and $P_2$.
\end{proof}

Note that using Lemma \ref{lem:rightfactor} gives a theoretical way to recursively obtain the list of the invariants describing the similarity classes of the irreducible factors of $P$: if $\mu$ is the smallest slope of the Newton polygon of $P$, then the $\mu$-reduction of $P$ can be factored in $k[T^\ell, \sigma^\ell]$, an irreducible right-factor of this $\mu$-reduction yields an irreducible right factor of $P$, and the data of $\mu$ and such a factor describe one class of irreducible polynomials in $K[T,\phi]$.

\bibliographystyle{plain}
\bibliography{biblio}
%-----------------------------------------

\end{document}